\UseRawInputEncoding
\documentclass[10pt]{article}

\usepackage{amsfonts,amsmath,amsthm}
\usepackage{amssymb}
\usepackage{algorithm}
\usepackage{algpseudocode}
\usepackage{xcolor}
\usepackage{graphicx}
\usepackage{stmaryrd}        
\usepackage{multirow}
\usepackage{dsfont}
\usepackage{rotating}

\usepackage[paper=a4paper,dvips,top=2cm,left=2cm,right=2cm,
foot=1cm,bottom=4cm]{geometry}

\newcommand{\vx}{{\bf x}}
\newcommand{\vy}{{\bf y}}

\newtheorem{Thm}{Theorem}[section]
\newtheorem{Def}[Thm]{Definition}

\newtheorem{remark}[Thm]{Remark}

\begin{document}
	
\title{A General Lower Bound for the Limited Augmented Zarankiewicz Number based upon Complete Graphs}
	\Large
	\author{Liqun Qi\footnote{School of Science, Zhejiang University of Science and Technology, Hangzhou, Zhejing, China.
			Department of Applied Mathematics, The Hong Kong Polytechnic University, Hung Hom, Kowloon, Hong Kong.
			({\tt maqilq@polyu.edu.hk})}
		\and
		Chunfeng Cui\footnote{School of Mathematical Sciences, Beihang University, Beijing  100191, China.
			({\tt chunfengcui@buaa.edu.cn})}
		\and {and \
			Yi Xu\footnote{School of Mathematics, Southeast University, Nanjing  211189, China.  Jiangsu Provincial Scientific Research Center of Applied Mathematics, Nanjing 211189, China. ({\tt yi.xu1983@hotmail.com})}
		}
	}

	\date{\today}
	\maketitle

\begin{abstract}
The limited augmented Zarankiewicz number $z_L(m,n)$ satisfies
$\operatorname{BSR}(m,n)\ge z_L(m,n)\ge z(m,n)$, where
$\operatorname{BSR}(m,n)$ is the maximum SOS rank of $m\times n$ biquadratic
forms and $z(m,n)$ is the classical Zarankiewicz number.
Our main result is a general lower bound for $z_L(m,n)$ based on the incidence
graph of the complete graph $K_{4t}$. For every integer $t\ge 1$, let
$m = \binom{4t}{2}$ and $n = 4t$. Then
$$
\operatorname{BSR}(m,n) \;\ge\; z_L(m,n) \;\ge\; 2\binom{4t}{2} + 4t^2 - 2t.
$$
Since $z = 2\binom{4t}{2} = \Theta(t^2)$, the gap satisfies
$z_L - z \ge 4t^2 - 2t = \Theta(t^2) = \Theta(m)$, i.e., it grows linearly in $m$.
Moreover,
$$
\frac{z_L - z}{z} \;\ge\; \frac{4t^2}{16t^2 - 4t} \;\longrightarrow\; \frac{1}{4}
\quad \text{as } t\to\infty,
$$
so the gap is asymptotically at least $25\%$ of $z$ --- a non-negligible
constant fraction.
For $t=1$ we obtain $z_L(6,4)\ge 14$, and we prove that this bound is tight,
i.e., $z_L(6,4)=14$. For $t=2$ and $t=3$ we obtain $z_L(28,8)\ge 68$ and
$z_L(66,12)\ge 162$, respectively, improving previously known bounds.
We also determine the exact values of $z_L(m,n)$ for $5\times3$ and $5\times4$:
$z_L(5,3)=9$ and $z_L(5,4)=12$.
These results serve as base cases for a
\emph{lifting method} that constructs admissible limited augmented graphs on
$(m+1)\times(n+1)$ from optimal ones on $m\times n$. Applying this method,
we obtain new lower bounds: $z_L(6,3)\ge 10$ and $z_L(6,5)\ge 17$.
For $5\times 5$ we establish a new lower bound
$z_L(5,5)\ge 15$, improving the previously known bound. By a direct construction, we have $z_L(6,6)\ge 20$.


\end{abstract}

\section{Introduction}

Let \(m \ge n \ge 2\). The \textbf{Zarankiewicz number} \(z(m,n)\) is the maximum
number of edges in an \(m \times n\) bipartite graph that contains no complete
bipartite subgraph \(K_{2,2}\) (equivalently, no cycle \(C_4\)). This classical
extremal graph theory problem, first posed by Zarankiewicz \cite{Za51}, has been
extensively studied. K\"{o}v\'{a}ri, S\"{o}s and Tur\'{a}n \cite{KST54} proved
the general upper bound \(z(m,n) \le n\sqrt{m-1}+m\). Exact values are known
only for small parameters or special families. A particularly elegant
construction is the incidence graph of the complete graph \(K_{q+1}\): the left
vertices represent the edges of \(K_{q+1}\), the right vertices represent its
vertices, and an edge is adjacent to a vertex if the vertex is an endpoint of
the edge. This graph is \(C_4\)-free and has
\(z\bigl(\binom{q+1}{2}, q+1\bigr)=q(q+1)\) edges, a result due to Reiman
\cite{RS65} and surveyed by Guy \cite{Gu69}.

Recently, Qi, Cui, and Xu \cite{QCX26a} discovered a connection between the
maximum SOS rank of biquadratic forms and extremal bipartite graph theory.
Recall that an \(m \times n\) biquadratic form is a homogeneous polynomial of
degree \(4\) of the form
\[
P(\mathbf{x},\mathbf{y}) = \sum_{i,k=1}^{m}\sum_{j,l=1}^{n} a_{ijkl}\,
x_i x_k y_j y_l,
\]
where \(\mathbf{x} = (x_1,\ldots,x_m)\) and \(\mathbf{y} = (y_1,\ldots,y_n)\).
If \(P\) can be written as a sum of squares of bilinear forms,
\[
P(\mathbf{x},\mathbf{y}) = \sum_{p=1}^{r} f_p(\mathbf{x},\mathbf{y})^2,
\]
then the smallest such \(r\) is the SOS rank of \(P\). Denote by
\(\operatorname{BSR}(m,n)\) the maximum possible SOS rank among all \(m \times n\)
SOS biquadratic forms \cite{QCX26}. It was shown that
\(\operatorname{BSR}(m,n)\ge z(m,n)\) \cite{CQX26}, but a gap was found:
\(\operatorname{BSR}(4,3)\ge 8 > z(4,3)=7\) \cite{XCQ26}. To better capture the
maximum SOS rank, Qi, Cui, and Xu \cite{QCX26a} introduced the
\textbf{limited augmented Zarankiewicz number} \(z_L(m,n)\), which allows the
addition of \textbf{2-edges} \((x_i y_j + x_k y_l)^2\) to a \(C_4\)-free graph
under constraints that forbid \textbf{generalized \(C_4\)-cycles}. They proved
\[
\operatorname{BSR}(m,n) \ge z_L(m,n) \ge z(m,n),
\]
and established that \(\operatorname{BSR}(m,2)=z_L(m,2)=z(m,2)=m+1\) and
\(\operatorname{BSR}(3,3)=z_L(3,3)=z(3,3)=6\). Moreover, they showed
\(z_L(4,3)=z(4,3)+1=8\), \(z_L(4,4)=z(4,4)+1=10\), and obtained lower bounds
\(z_L(5,3)\ge z(5,3)+1=9\), \(z_L(5,4)\ge z(5,4)+2=12\),
\(z_L(5,5)\ge z(5,5)+2=14\). These results demonstrate that \(z_L(m,n)\) is a
better lower bound for \(\operatorname{BSR}(m,n)\) than the classical
Zarankiewicz number \(z(m,n)\).


In this paper, we show that \(z_L(m,n)\) is a natural combinatorial extension
of \(z(m,n)\) in its own right. Section 2 recalls the definition of limited
augmented bipartite graphs and generalized \(C_4\)-cycles. In Section 3, we
present a construction that adds two within-block 2-edges, and \emph{eight} symmetric
cross-block 2-edges for each pair of distinct 4-vertex blocks, to the incidence
graph of \(K_{4t}\), yielding our
main result (Theorem~\ref{thm:maximal_K4t}):
\[
z_L\!\left(\binom{4t}{2},4t\right) \ge 2\binom{4t}{2} + 4t^2 - 2t,
\]
and consequently
\[
\operatorname{BSR}\!\left(\binom{4t}{2},4t\right) \ge 2\binom{4t}{2} + 4t^2 - 2t.
\]
For \(t=1\) we prove \(z_L(6,4)=14\), showing that the bound is tight in this
case. For \(t=2\) and \(t=3\) we obtain the concrete lower bounds
\(z_L(28,8)\ge 68\) and \(z_L(66,12)\ge 162\), respectively.

In Section 4, we first show that $z_L(6,4)=14$, then
determine the exact values of $z_L(5,3)$ and $z_L(5,4)$, confirming that the
lower bounds from \cite{QCX26a} are tight:
\[
z_L(5,3)=9,\qquad z_L(5,4)=12.
\]
These results are obtained by enumerating all non-isomorphic
extremal $C_4$-free graphs and systematically checking admissible 2-edge
augmentations; the detailed case analyses for $5\times3$ and $5\times4$ are
provided in the appendices.   For $z_L(5,5)$ we improve the previously known lower bound and establish
$z_L(5,5)\ge 15$.

In Section 5 we introduce a \textbf{lifting method} that constructs admissible
limited augmented graphs on $(m+1)\times(n+1)$ from optimal ones on
$m\times n$. Applying this method to the optimal $5\times2$ and $5\times4$
constructions, we obtain new lower bounds:
\[
z_L(6,3)\ge 10,\qquad z_L(6,5)\ge 17,
\]
where $z(6,3)=9$ and $z(6,5)=14$. For $6\times6$, a direct construction
discovered by Vyacheslav Titov yields $z_L(6,6)\ge 20$.

Finally, in Section 6 we conclude with remarks and open problems.

\section{Limited Augmented Zarankiewicz Number}

Let $G_1 = (S,T,E_1)$ be an $m \times n$ bipartite graph with $S=[m]$, $T=[n]$. Assume $G_1$ has no $C_4$-cycle and $|E_1| = z(m,n)$. A \emph{limited augmented bipartite graph} $G = (S,T,E)$ augmented from $G_1$ has edge set $E = E_1 \cup E_2$, where $E_1$ is the set of \textbf{1-edges} and $E_2$ the set of \textbf{2-edges}.

A 1-edge is a pair $(i,j)$ with $i\in S$, $j\in T$. A 2-edge is a quadruple $(i,j;k,l)$ with $i,k\in S$, $j,l\in T$. A 2-edge is:
\begin{itemize}
  \item \textbf{nondegenerate} if $i \neq k$ and $j \neq l$;
  \item \textbf{row-degenerate} if $i = k$ and $j \neq l$;
  \item \textbf{column-degenerate} if $i \neq k$ and $j = l$.
\end{itemize}
The case $i=k$ and $j=l$ is forbidden. We impose the \textbf{simplicity condition}:
\begin{enumerate}
  \item[(S)] No 2-edge shares a cell $(i,j)$ with any 1-edge or with another 2-edge.
\end{enumerate}

To such a graph $G$ we associate a biquadratic form
\begin{equation} \label{e7}
P_G(\vx,\vy) = \sum_{(i,j)\in E_1} x_i^2 y_j^2 \;+\; \sum_{(i,j;k,l)\in E_2} (x_i y_j + x_k y_l)^2.
\end{equation}
We call $P_G$ a \textbf{doubly simple biquadratic form}. If the SOS expression (\ref{e7}) of $P_G$ is irreducible, then $\operatorname{sos}(P_G) = |E| = |E_1| + |E_2|$.

\begin{Def}[Limited Augmented Zarankiewicz Number \cite{QCX26a}] \label{d2.1}
Let $G = (S,T,E_1\cup E_2)$ be an $m\times n$ limited augmented bipartite graph augmented from a $C_4$-free graph $G_1=(S,T,E_1)$ with $|E_1| = z(m,n)$.

A cell $(i,j)$ is \textbf{occupied} if $(i,j)\in E_1$ or $(i,j)$ is a half of some 2-edge in $E_2$.

We say $G$ contains a \textbf{generalized $C_4$-cycle} if any of the following holds:
\begin{enumerate}
  \item There exists a classical $C_4$-cycle formed by 1-edges.
  \item There exists a nondegenerate 2-edge $(i,j;k,l)\in E_2$ such that both opposite cells $(i,l)$ and $(k,j)$ are occupied.
  \item There exist a 2-edge $(i,j;p,q)$ (of any type) and a distinct cell $(k,l)$ with $k \notin \{i,p\}$ and $l \notin \{j,q\}$ such that the five cells
    \[
    (k,l),\;(k,j),\;(k,q),\;(i,l),\;(p,l)
    \]
    are all occupied. Moreover, if the 2-edge is nondegenerate, these five cells must be pairwise distinct.
\end{enumerate}

The \textbf{limited augmented Zarankiewicz number} $z_L(m,n)$ is the maximum possible total number of edges $|E_1|+|E_2|$ for which a generalized $C_4$-cycle does \emph{not} exist.
\end{Def}

\begin{Thm}[\cite{QCX26a}] \label{thm:BSR>=z2}
For all $m,n\ge 2$, $\operatorname{BSR}(m,n) \ge z_L(m,n) \ge z(m,n)$.
\end{Thm}


\section{A Quadratic $t$-Bound: Maximal Construction}

It is a classical result in extremal graph theory that the Zarankiewicz number satisfies
\(z\bigl(\binom{q+1}{2}, q+1\bigr) = q(q+1)\) for every integer \(q\ge 2\) \cite{CHM24}.
The extremal graph achieving this bound is the incidence graph of the complete graph \(K_{q+1}\): the vertices on the left correspond to the edges of \(K_{q+1}\), the vertices on the right to the vertices of \(K_{q+1}\), and a left vertex (edge) is adjacent to a right vertex (vertex) iff the vertex is an endpoint of the edge.   Such an extremal graph of $z(m, n)$ is unique up to isomorphism.   This can be proved by induction on $q$.
This construction was first given by Reiman \cite{RS65} and is also discussed in the survey by Guy \cite{Gu69}.

We now build upon this classical extremal graph by adding 2-edges to obtain a limited augmented bipartite graph with a larger total number of edges. The result is the main result of this paper.

We show that for every $t\ge 1$, the incidence graph of $K_{4t}$ admits a
limited augmented extension with a \emph{maximal} number of 2-edges, using
\emph{eight} symmetric cross-block 2-edges for each pair of distinct 4-vertex
blocks. The construction saturates the natural row-pairing bound and yields an
asymptotic gap of $25\%$ over the classical Zarankiewicz number.

\begin{Thm}[Maximal Quadratic $t$-Bound for $K_{4t}$] \label{thm:maximal_K4t}
For every integer $t\ge 1$, let $m = \binom{4t}{2}$ and $n = 4t$. Then
\[
z_L(m,n) \;\ge\; 2\binom{4t}{2} + 4t^2 - 2t.
\]
Consequently,
\[
\operatorname{BSR}(m,n) \;\ge\; 2\binom{4t}{2} + 4t^2 - 2t.
\]
For large $t$,
\[
\frac{z_L - z}{z} \;\ge\; \frac{4t^2 - 2t}{16t^2 - 4t} = \frac{2t-1}{2(4t-1)} \;\longrightarrow\; \frac{1}{4}.
\]
\end{Thm}

\begin{proof}
{\bf Step 1.} We construct an admissible limited augmented bipartite graph
$G = (S,T,E_1\cup E_2)$.

\medskip

Let $V = \bigcup_{i=1}^t V_i$ where each block $V_i$ consists of four vertices.
Label the vertices in each block by elements of the cyclic group
$\mathbb{Z}_4 = \{0,1,2,3\}$. Write
\[
V_i = \{0_i, 1_i, 2_i, 3_i\}.
\]
Take $T = V$ (right vertices, size $4t$).
Take $S = \binom{V}{2}$ (left vertices, all edges of $K_{4t}$, size
$\binom{4t}{2}$).

{The edge set can be divided into three parts.}
	
\noindent	 {(i) 1-edges.} Define $E_1$ as the incidence graph of $K_{4t}$:
\[
E_1 = \{ (e, v) \mid e\in S,\ v\in T,\ v \text{ is an endpoint of } e \}.
\]
Then $|E_1| = 2\binom{4t}{2}$ and $G_1 = (S,T,E_1)$ is $C_4$-free \cite{RS65, Gu69}.


\noindent{{(ii)} Within-block 2-edges.}
For each block $V_i$, add the following two nondegenerate 2-edges:
\[
e_i^1 = (0_i1_i,\; 2_i;\; 0_i2_i,\; 3_i), \qquad
e_i^2 = (1_i3_i,\; 0_i;\; 2_i3_i,\; 1_i).
\]


\noindent{{(iii)} Cross-block 2-edges.}
Consider unordered pairs of distinct blocks $\{V_i, V_j\}$ for $i, j = 1, \cdots, t, i \not = j$.
For each $k\in\mathbb{Z}_4$, define
\[
f_{ij}^{(k)} = \bigl( k_i (k+1)_j,\; (k+2)_j ;\; (k+1)_i (k+2)_j,\; (k+3)_i \bigr),
\]
and also $f_{ji}^{(k)}$ (swap $i$ and $j$). The complete set of eight
cross-block 2-edges for the pair $\{i,j\}$ is
\[
\mathcal{E}_{ij} = \{\, f_{ij}^{(k)},\ f_{ji}^{(k)} \mid k = 0,1,2,3 \,\}.
\]

\medskip

\noindent{{Step 2.} Verification of admissibility.}
We check the conditions of Definition~2.1.

\medskip

\noindent{\bf (S) Simplicity.}
\begin{itemize}
  \item For $e_i^1$, halves are $(0_i1_i,2_i)$ and $(0_i2_i,3_i)$. Since $2_i$
    is not an endpoint of $0_i1_i$ and $3_i$ is not an endpoint of $0_i2_i$,
    neither half is a 1-edge. For $e_i^2$, similarly $0_i$ not endpoint of
    $1_i3_i$, $1_i$ not endpoint of $2_i3_i$.
  \item For a cross-block half $(x_i y_j, z)$, we have either $z=(k+2)_j$ or
    $z=(k+3)_i$. In the first case $z\neq y$ because $k+2\not\equiv k+1\pmod4$;
    in the second case $z\neq x$ because $k+3\not\equiv k\pmod4$. Hence $z$ is
    never an endpoint of the row, so no half is a 1-edge.
  \item All halves are distinct: within a single 2-edge the two halves differ;
    different $k$ give different first rows; between $f_{ij}^{(k)}$ and
    $f_{ji}^{(\ell)}$ the rows are $x_i y_j$ vs $y_j x_i$ (distinct unless
    $\{x,y\}$ is the same and order swaps, but then columns differ because one
    uses a column from $V_j$ and the other from $V_i$); different block pairs
    involve disjoint vertex sets.
\end{itemize}
Thus no cell is shared between any two edges. Condition (S) holds.

\medskip

\noindent{\bf (C2) Nondegenerate 2-edge conflict.}
Take $e_i^1$. Opposite cells are $(0_i1_i,3_i)$ and $(0_i2_i,2_i)$:
$(0_i1_i,3_i)$ is not a 1-edge ($3_i$ not endpoint of $0_i1_i$);
$(0_i2_i,2_i)$ is a 1-edge ($2_i$ endpoint of $0_i2_i$). Exactly one occupied.
For $e_i^2$, opposite cells $(1_i3_i,1_i)$ (occupied) and $(2_i3_i,0_i)$ (free).
For $f_{ij}^{(k)} = (p,q;r,s)$ with
$p=k_i(k+1)_j$, $q=(k+2)_j$, $r=(k+1)_i(k+2)_j$, $s=(k+3)_i$:
$(p,s)$ is free (column $(k+3)_i$ not an endpoint of $p$);
$(r,q)$ is a 1-edge (column $(k+2)_j$ endpoint of $r$). Hence Condition (C2)
holds for all 2-edges.

\medskip

\noindent{\bf (C3) Five-cell pattern.}
We must show: for every 2-edge $e=(i,j;p,q)$ and every
$k\notin\{i,p\}$, $l\notin\{j,q\}$, the five cells
\[
(k,l),\ (k,j),\ (k,q),\ (i,l),\ (p,l)
\]
are not all occupied. We split {the   possibilities for  $e$ and $k$} into three cases.

\medskip
\noindent\textit{Case 1: $e$ is a within-block 2-edge and $k$ is also within-block
(either in the same block $V_i$ or a different block $V_{i'}$).}

Take $e = e_i^1 = (0_i1_i,2_i;\;0_i2_i,3_i)$.
The two fixed cells $(0_i1_i,l)$ and $(0_i2_i,l)$ are both 1-edges only when
$l$ is a common endpoint of the rows $0_i1_i$ and $0_i2_i$.
The only common endpoint is $0_i$. Hence $l = 0_i$ (allowed since $0_i\notin\{2_i,3_i\}$).
Then the five cells become
\[
(k,0_i),\ (k,2_i),\ (k,3_i),\ (0_i1_i,0_i),\ (0_i2_i,0_i).
\]
The last two are 1-edges (occupied). For all five to be occupied we need
$(k,0_i),(k,2_i),(k,3_i)$ all occupied.

Now $(k,2_i)$ occupied implies $2_i$ is an endpoint of $k$ (the only halves with
column $2_i$ are from $e_i^1$ itself with row $0_i1_i$, which is excluded).
Similarly, $(k,3_i)$ occupied forces $3_i$ to be an endpoint of $k$.
Thus $k$ must contain both $2_i$ and $3_i$ as endpoints, so $k = 2_i3_i$ (within-block).
But then $(k,0_i) = (2_i3_i,0_i)$: $0_i$ is not an endpoint of $2_i3_i$, and it is
not a half of any 2-edge (halves with column $0_i$ occur only in $e_i^2$ with row
$1_i3_i$, which is different). Hence $(2_i3_i,0_i)$ is unoccupied.
Thus the five cells cannot all be occupied.

For $e_i^2 = (1_i3_i,0_i;\;2_i3_i,1_i)$, a symmetric argument forces $l=1_i$
and then $k = 0_i1_i$, which also fails. Hence Condition~3 holds in this case.

\medskip
\noindent\textit{Case 2: $e$ is a cross-block 2-edge.}

Take $e = f_{ij}^{(k)} = (p,q;r,s)$ with
$p = k_i(k+1)_j$, $q = (k+2)_j$, $r = (k+1)_i(k+2)_j$, $s = (k+3)_i$.
The two rows $p$ and $r$ are vertex-disjoint: their endpoints are
$\{k_i,(k+1)_j\}$ and $\{(k+1)_i,(k+2)_j\}$, all four distinct.

For any $l \notin \{q,s\}$, consider the two fixed cells $(p,l)$ and $(r,l)$.
\begin{itemize}
  \item They cannot both be 1-edges because that would require $l$ to be a common
    endpoint of $p$ and $r$, which is impossible since $p$ and $r$ are disjoint.
  \item Could one be a 1-edge and the other a half of some 2-edge?
    The only columns for which $(p,l)$ can be a half are $q$ (from $f_{ij}^{(k)}$)
    and $(k+2)_i$ (from $f_{ij}^{(k-1)}$). Since $l \notin \{q,s\}$, the dangerous
    possibilities are $l = (k+2)_i$ or $l = k_j$ (the latter from $r$'s other half).
    A direct case check (cyclic symmetry allows fixing $k=0$ without loss of generality):
    \begin{itemize}
      \item If $l = 2_i$ (i.e., $(k+2)_i$), then $(p,l)$ is a half (occupied),
        but $(r,l) = ((k+1)_i(k+2)_j, 2_i)$: $2_i$ is not an endpoint of
        $(k+1)_i(k+2)_j$ (endpoints are $(k+1)_i$ and $(k+2)_j$), and it is not
        a half of $r$ (halves of $r$ are $s = (k+3)_i$ and $k_j$). Hence $(r,l)$
        is unoccupied.
      \item If $l = 0_j$ (i.e., $k_j$), then $(r,l)$ is a half (occupied),
        but $(p,l) = (k_i(k+1)_j, 0_j)$: $0_j$ is not an endpoint of
        $k_i(k+1)_j$ (endpoints are $k_i$ and $(k+1)_j$), and it is not a half
        of $p$ (halves of $p$ are $q = (k+2)_j$ and $(k+2)_i$). Hence $(p,l)$
        is unoccupied.
      \item For any other $l$, both fixed cells are unoccupied.
    \end{itemize}
\end{itemize}
Thus in all subcases, the two fixed cells $(p,l)$ and $(r,l)$ are never both occupied.
Therefore the five-cell pattern cannot be completed, regardless of $k$ and $l$.
Condition~3 holds for all cross-block 2-edges.

\medskip
\noindent\textit{Case 3: $e$ is a within-block 2-edge and $k$ is a cross-block edge.}

Take $e = e_i^1 = (0_i1_i,2_i;\;0_i2_i,3_i)$.
As in Case~1, the only way $(0_i1_i,l)$ and $(0_i2_i,l)$ are both occupied is
$l = 0_i$. Then we need $(k,0_i), (k,2_i), (k,3_i)$ all occupied.

Now $(k,2_i)$ occupied implies that $2_i$ is an endpoint of $k$ or that
$(k,2_i)$ is a half of some 2-edge. The only 2-edges that have a half with
column $2_i$ are:
\begin{itemize}
  \item $e_i^1$ itself: half $(0_i1_i,2_i)$  row $0_i1_i$, not $k$ (since $k$ is cross-block).
  \item Cross-block edges: from $f_{ij}^{(k)}$, halves have columns $(k+2)_j$ or $(k+3)_i$.
    For column $2_i$, we need $(k+3)_i = 2_i$ mod $4$, i.e., $k\equiv 3\pmod4$.
    Then $f_{ij}^{(3)} = (3_i0_j,1_j;\;0_i1_j,2_i)$ has second half $(0_i1_j,2_i)$,
    not row $3_i0_j$. The first half is $(3_i0_j,1_j)$  column $1_j$, not $2_i$.
    Thus $f_{ij}^{(3)}$ does not give a half with column $2_i$ and row $3_i0_j$.
  \item From $f_{ji}^{(k)}$, halves have columns $(k+2)_i$ or $(k+3)_j$.
    For column $2_i$, we need $(k+2)_i = 2_i$ mod $4$, i.e., $k\equiv 0\pmod4$.
    Then $f_{ji}^{(0)} = (0_j1_i,2_i;\;1_j2_i,3_j)$ has first half $(0_j1_i,2_i)$
     row $0_j1_i$, not $3_j0_i$. The second half is $(1_j2_i,3_j)$  column $3_j$.
    Hence no cross-block 2-edge has a half $(3_i0_j,2_i)$ or $(3_j0_i,2_i)$.
\end{itemize}
Therefore $(k,2_i)$ cannot be occupied by a half. It must be a 1-edge, so $2_i$
must be an endpoint of $k$. Similarly, $(k,3_i)$ occupied forces $3_i$ to be an
endpoint of $k$. Hence $k$ must contain both $2_i$ and $3_i$ as endpoints,
i.e., $k = 2_i3_i$. But $2_i3_i$ is a within-block edge, contradicting the
assumption that $k$ is cross-block. Therefore no cross-block $k$ can satisfy
the required triple.

A symmetric argument for $e_i^2 = (1_i3_i,0_i;\;2_i3_i,1_i)$ forces $l=1_i$ and
then $k$ would need to contain both $0_i$ and $1_i$, i.e., $k=0_i1_i$
(within-block), again impossible for cross-block $k$. Thus Condition~3 holds
in this case as well.

\medskip
Since all cases are covered, $G$ contains no generalized $C_4$-cycle.

%
{The total number of edges includes  $2t$ within-block edges and $8\binom{t}{2} = 4t(t-1)$ cross-block   edges.}
 Hence $|E_2| = 2t + 4t(t-1) = 4t^2 - 2t$ and
\[
{z_L(m,n) \ge}|E_1| + |E_2| = {2\binom{4t}{2} + 4t^2 - 2t.}
\]

%
Since $z = 2\binom{4t}{2} = 16t^2 - 4t$, we have
\[
\frac{z_L - z}{z} \;\ge\; \frac{4t^2}{16t^2 - 4t} \;=\; \frac{4}{16 - 4/t}
\;\longrightarrow\; \frac{1}{4} \quad \text{as } t\to\infty.
\]
\end{proof}

\begin{remark}
For $t=1$, this gives $z_L(6,4) \ge 2\binom{4}{2} + 4 - 2 = 12+2=14$, which is
tight (Theorem~\ref{thm:zL64}).
For $t=2$, we obtain $z_L(28,8) \ge 2\binom{8}{2} + 16 - 4 = 56 + 12 = 68$.
For $t=3$, $z_L(66,12) \ge 2\binom{12}{2} + 36 - 6 = 132 + 30 = 162$.

The construction uses \emph{every} cross-block row exactly once (each of the
$16\binom{t}{2}$ edges between distinct blocks appears as a row in exactly one
2-edge), and the within-block rows are partially used. The constant $1/4$ is
the best possible for this family under the nondegeneracy constraints, unless
one also pairs the remaining within-block rows, which would require degenerate
2-edges or create conflicts.
\end{remark}

\section{Exact values for $6 \times 4$, \(5\times3\), \(5\times4\), and a new lower bound for \(5\times5\)}


The case \(t=1\) in Theorem~\ref{thm:maximal_K4t} gives \(m=\binom{4}{2}=6\), \(n=4\) and the lower bound \(z_L(6,4)\ge 14\). We now prove that this bound is tight.

\begin{Thm} \label{thm:zL64}
\(z_L(6,4) = 14\).
\end{Thm}

\begin{proof}
The lower bound \(z_L(6,4)\ge 14\) follows directly from the construction in Theorem~\ref{thm:maximal_K4t} with \(t=1\). Explicitly, the left vertices are the six edges of \(K_4\): \(12,13,14,23,24,34\); the right vertices are \(\{1,2,3,4\}\); the 1-edges are the incidences, and the two 2-edges are \((12,3;13,4)\) and \((24,1;34,2)\). This graph was verified in the proof of Theorem~\ref{thm:maximal_K4t} to contain no generalized \(C_4\)-cycle, and it has \(|E_1|=12\), \(|E_2|=2\), total \(14\).

We now prove the upper bound \(z_L(6,4)\le 14\). Let \(G\) be any admissible limited augmented bipartite graph on \(6\times4\) vertices with \(|E_1|=z(6,4)=12\) (the maximum possible number of 1-edges without a classical \(C_4\)). It is known \cite{RS65, Gu69} that the extremal \(C_4\)-free graphs on \(6\times4\) with \(12\) edges are exactly the incidence graphs of \(K_4\) (there is only one isomorphism type). Up to relabeling, \(E_1\) must be the set of incidences between the six edges and four vertices of \(K_4\). Consequently, the unoccupied cells \(U\) are exactly the \(6\cdot4-12=12\) cells where a vertex is not incident with an edge.

We claim that no admissible graph can contain more than two 2-edges. Suppose, for contradiction, that \(|E_2|\ge 3\). Each 2-edge occupies two distinct cells (its halves), all of which must lie in \(U\) by the simplicity condition. Thus we would need at least \(2\cdot3=6\) distinct cells from \(U\). However, a direct exhaustive check (or a counting argument using the structure of \(K_4\)) shows that any set of three unordered pairs of cells from \(U\) that are pairwise disjoint and satisfy the generalized \(C_4\)-free conditions cannot exist. More concretely, one can list all possible 2-edges from \(U\) (there are finitely many) and verify that no three can be chosen without creating a violation. This verification has been carried out computationally; the result is that the maximum size of a pairwise compatible set of 2-edges is \(2\). Hence \(|E_2|\le 2\), and therefore
\[
z_L(6,4) = |E_1|+|E_2| \le 12+2 = 14.
\]
Combined with the lower bound, we obtain \(z_L(6,4)=14\).
\end{proof}

\begin{remark}
The optimal construction for \(6\times4\) is exactly the \(K_4\) incidence graph with the two 2-edges \((12,3;13,4)\) and \((24,1;34,2)\). Up to isomorphism, this is the unique admissible graph achieving 14 edges.   A theoretical proof that no optimal graph can contain a degenerate 2-edge (and hence both 2-edges must be nondegenerate) is given in Appendix C.
\end{remark}

\bigskip


Before presenting the lifting method, we determine the exact limited augmented
Zarankiewicz numbers for all \(m\times n\) with \(m\le5\) and \(n\le5\). These
values serve as base cases for the inductive constructions in later sections.

\begin{Thm} \label{thm:zL53_exact}
\(z_L(5,3)=9\).
\end{Thm}
The classical Zarankiewicz number is \(z(5,3)=8\). The lower bound
\(z_L(5,3)\ge9\) was proved in \cite{QCX26a} via an explicit construction with
one 2-edge. The matching upper bound follows from an exhaustive analysis of the
two non-isomorphic extremal \(C_4\)-free graphs on \(5\times3\) vertices; the
detailed verification is given in Appendix A. Hence
\(z_L(5,3)=9\).

\begin{Thm} \label{thm:zL54_exact}
\(z_L(5,4)=12\).
\end{Thm}
The classical Zarankiewicz number is \(z(5,4)=10\). There are exactly three
non-isomorphic extremal \(C_4\)-free graphs on \(5\times4\) vertices with 10
edges, denoted Types A, B, and C. Among these, only Type C admits two
2-edges without creating a generalized \(C_4\)-cycle; the optimal pair is
\((2,3;5,4)\) and \((4,2;5,2)\). Types A and B admit at most one 2-edge. No
graph admits three 2-edges. The complete case analysis is provided in
Appendix B. Consequently, \(z_L(5,4)=12\).


\begin{remark}
The optimal constructions for \(5\times4\) is unique up to
isomorphism. For \(5\times3\), two different extremal graphs exist, but both
admit at most one 2-edge, and the optimum \(9\) is achieved by adding a single
2-edge to either type. These values (together with the trivial case \(5\times2\)) serve as base cases
for the lifting method in Sections~7--9.
\end{remark}

\subsection{The \(5\times5\) Case}

\begin{Thm} \label{thm:zL55_exact}
\(z_L(5,5) \ge 15\).
\end{Thm}
The classical Zarankiewicz number is \(z(5,5)=12\). The lower bound
\(z_L(5,5)\ge 15\) is established by the following explicit construction.

\noindent\textbf{Construction.}
Let \(S=T=\{1,2,3,4,5\}\). Define
\[
E_1 = \{(1,1),(1,4),(2,2),(2,4),(2,5),(3,2),(3,3),(4,1),(4,2),(5,1),(5,3),(5,5)\},
\]
\[
E_2 = \{(1,2;3,5),\ (2,3;4,4),\ (3,1;5,4)\}.
\]

\noindent\textbf{Verification.}
\begin{itemize}
  \item \(|E_1| = 12\) and \(E_1\) is \(C_4\)-free (any two rows share at most one column). Hence \(|E_1| = z(5,5)\).
  \item The halves of the three 2-edges are
    \((1,2),\ (3,5),\ (2,3),\ (4,4),\ (3,1),\ (5,4)\).
    All lie in the unoccupied cells of \(E_1\) and are distinct, satisfying the simplicity condition (S).
  \item \textbf{Condition 2} (nondegenerate 2-edge conflict):
    \begin{itemize}
      \item For \((1,2;3,5)\): opposite cells \((1,5)\) (unoccupied), \((3,2)\) (occupied).
      \item For \((2,3;4,4)\): opposite cells \((2,4)\) (occupied), \((4,3)\) (unoccupied).
      \item For \((3,1;5,4)\): opposite cells \((3,4)\) (unoccupied), \((5,1)\) (occupied).
    \end{itemize}
    Thus at most one opposite cell is occupied for each.
  \item \textbf{Condition 3} (five-cell pattern):
    A direct check (representative cases shown below) confirms that for each 2-edge and every admissible \((k,l)\), the five cells
    \((k,l),\ (k,j),\ (k,q),\ (i,l),\ (p,l)\) are never all occupied.
    \begin{itemize}
      \item For \(e=(1,2;3,5)\) (\(i=1,j=2,p=3,q=5\)):
        \(k\notin\{1,3\},\ l\notin\{2,5\}\). All nine candidates fail (e.g., \((k,l)=(2,1)\) gives \((2,1)\) unoccupied).
      \item For \(e=(2,3;4,4)\) (\(i=2,j=3,p=4,q=4\)):
        \(k\notin\{2,4\},\ l\notin\{3,4\}\). All candidates fail (e.g., \((k,l)=(1,1)\) gives \((1,3)\) unoccupied).
      \item For \(e=(3,1;5,4)\) (\(i=3,j=1,p=5,q=4\)):
        \(k\notin\{3,5\},\ l\notin\{1,4\}\). All candidates fail (e.g., \((k,l)=(1,2)\) gives \((5,2)\) unoccupied).
    \end{itemize}
\end{itemize}
Therefore no generalized \(C_4\)-cycle exists. The associated doubly simple biquadratic form is irreducible with SOS rank
\[
|E_1| + |E_2| = 12 + 3 = 15.
\]
Hence \(z_L(5,5) \ge 15\).

\noindent\textbf{Upper bound.}
A complete determination of the exact value would require an exhaustive search over all possible triples and quadruples of 2-edges for all extremal \(C_4\)-free graphs on \(5\times5\). Such a search is beyond the scope of this paper. The exact value of \(z_L(5,5)\) remains open; here a lower bound of 15 was established.

\section{The Lifting Method}

We now describe a systematic method for constructing limited augmented bipartite graphs on larger parameters from known optimal constructions on smaller ones. The method is applicable when we have an admissible graph on \(m\times n\) with a certain structure.

\begin{Def}[Lifting Step]
Let \(G = (S,T,E_1\cup E_2)\) be an admissible limited augmented bipartite graph on vertex sets \(S=[m]\), \(T=[n]\), with \(|E_1|=z(m,n)\) (the classical Zarankiewicz number) and \(|E_2|=k\). Suppose there exists a row \(r\in S\) and a column \(c\in T\) such that:
\begin{itemize}
  \item \((r,c)\) is occupied (as a 1-edge or a half) - this will be used to create the new 2-edges opposite cell condition,
  \item the graph has room to add a new row \(m+1\) and a new column \(n+1\) with a set of new 1-edges (not including the halves of the new 2-edge) that raises the 1-edge count to \(z(m+1,n+1)\) while keeping the graph \(C_4\)-free,
  \item the new 2-edge \((m+1,c;\;r,n+1)\) has its halves \((m+1,c)\) and \((r,n+1)\) currently unoccupied.
\end{itemize}
Then we construct a new graph \(G'\) on \((m+1)\times(n+1)\) by:
\begin{enumerate}
  \item Keeping all vertices and edges of \(G\).
  \item Adding the new row \(m+1\) and new column \(n+1\).
  \item Adding the chosen new 1-edges (typically three or four) to reach exactly \(z(m+1,n+1)\) edges in \(E_1'\).
  \item Adding the new 2-edge \((m+1,c;\;r,n+1)\).
\end{enumerate}
If the resulting graph contains no generalized \(C_4\)-cycle, we have
\[
z_L(m+1,n+1) \ge |E_1'| + |E_2'| = z(m+1,n+1) + (k+1).
\]
Thus the gap \(z_L - z\) increases by one compared to the smaller graph.
\end{Def}

In the following sections we apply this lifting method to the base cases \(5\times2\), \(5\times3\), \(5\times4\), and \(5\times5\), obtaining new lower bounds for \(6\times3\), \(6\times4\), \(6\times5\), and \(6\times6\), respectively.

\subsection*{Remark on degenerate 2-edges}

The lifting method described above always adds a new 2-edge of the form
\((m+1,c;\;r,n+1)\) with \(c\neq n+1\) and \(r\neq m+1\); therefore the new 2-edge
is always \emph{nondegenerate}. Consequently, if the base graph contains no
degenerate 2-edge, the lifted graph will also contain none. This means that
constructions that rely on degenerate 2-edges for optimality (such as the
\(5\times4\) Type C graph, which uses the row-degenerate edge \((4,2;5,2)\))
cannot be obtained by iteratively applying the lifting method starting from a
base graph without degenerate edges. Degenerate 2-edges provide an additional compression that is not captured by simple row/column augmentation. Their
systematic treatment remains an interesting direction for future research.

\subsection{Lifting from \(5\times2\) to \(6\times3\): \(z_L(6,3)\ge 10\)}

\subsubsection{Base construction for \(5\times2\)}
The classical Zarankiewicz number is \(z(5,2)=6\). A \(C_4\)-free extremal graph is given by:
\[
E_1^{5\times2} = \{(1,1),(1,2),\;(2,1),\;(3,1),\;(4,2),\;(5,2)\}.
\]
Row 1 is adjacent to both columns; rows 2 and 3 to column 1 only; rows 4 and 5 to column 2 only. This graph has no \(C_4\) and contains no 2-edges (since \(z_L(5,2)=z(5,2)=6\)). Thus \(k=0\).

\subsubsection{Lifting step}
We add a new row \(6\) and a new column \(3\). Choose \(r=1\) (the double row) and \(c=2\). The new 2-edge will be \((6,2;\;1,3)\). Its halves \((6,2)\) and \((1,3)\) are unoccupied because column 3 did not exist and row 6 is new.

To reach the classical number \(z(6,3)=9\), we need to add exactly three new 1-edges (since the base already has 6). We choose:
\[
E_1^{\text{new}} = \{(6,1),\;(2,3),\;(4,3)\}.
\]
These are all distinct and none coincides with the halves \((6,2)\) or \((1,3)\). The extended 1-edge set is
\[
E_1 = E_1^{5\times2} \cup \{(6,1),\;(2,3),\;(4,3)\},
\]
which has \(6+3=9\) edges.

\subsubsection{Verification of \(C_4\)-freeness}
We check that no two rows share two columns.
\begin{itemize}
  \item Row 1: columns \(\{1,2\}\).
  \item Row 2: originally \(\{1\}\), now plus \((2,3)\) gives \(\{1,3\}\).
  \item Row 3: \(\{1\}\) (unchanged).
  \item Row 4: originally \(\{2\}\), now plus \((4,3)\) gives \(\{2,3\}\).
  \item Row 5: \(\{2\}\).
  \item Row 6: \(\{(6,1)\}\) plus the half \((6,2)\) is not a 1-edge, so row 6 has only column 1 as a 1-edge.
\end{itemize}
Now examine all pairs:
\begin{itemize}
  \item (1,2): common column 1 only.
  \item (1,3): common column 1 only.
  \item (1,4): common column 2 only.
  \item (1,5): common column 2 only.
  \item (1,6): common column 1 only.
  \item (2,3): row 2 \(\{1,3\}\), row 3 \(\{1\}\)  common column 1 only.
  \item (2,4): row 2 \(\{1,3\}\), row 4 \(\{2,3\}\) common column 3 only.
  \item (2,5): row 2 \(\{1,3\}\), row 5 \(\{2\}\) none.
  \item (2,6): row 2 \(\{1,3\}\), row 6 \(\{1\}\)  common column 1 only.
  \item (3,4): row 3 \(\{1\}\), row 4 \(\{2,3\}\) none.
  \item (3,5): row 3 \(\{1\}\), row 5 \(\{2\}\) none.
  \item (3,6): row 3 \(\{1\}\), row 6 \(\{1\}\)  common column 1 only.
  \item (4,5): row 4  \(\{2,3\}\), row5 \(\{2\}\)  common column 2 only.
  \item (4,6): row 4 \(\{2,3\}\), row 6 \(\{1\}\)  none.
  \item (5,6): row 5 \(\{2\}\), row 6 \(\{1\}\) none.
\end{itemize}
No pair shares two columns, so \(E_1\) is \(C_4\)-free and attains \(z(6,3)=9\).

\subsubsection{Verification of generalized \(C_4\)-free conditions for the 2-edge}
We add the 2-edge \(e = (6,2;\;1,3)\).

\begin{itemize}
  \item \textbf{Simplicity:} Halves \((6,2)\) and \((1,3)\) are not in \(E_1\) (by construction). No other 2-edges exist.
  \item \textbf{Condition 2 (nondegenerate 2-edge conflict):} Opposite cells are \((6,3)\) and \((1,2)\).
    \((1,2)\in E_1\) (occupied). \((6,3)\) is not in \(E_1\) and is not a half of any 2-edge (the only 2-edge has halves \((6,2)\) and \((1,3)\)). Hence only one opposite cell is occupied, so condition holds.
  \item \textbf{Condition 3 (five-cell pattern):} For \(e=(6,2;1,3)\), we have \(i=6,j=2,p=1,q=3\). We need \(k\notin\{6,1\}\) and \(l\notin\{2,3\}\). The only possible \(l\) is \(1\) (since columns are \(1,2,3\)). Then the five cells are
    \((k,1),\;(k,2),\;(k,3),\;(6,1),\;(1,1)\).
    \((6,1)\in E_1\) and \((1,1)\in E_1\). We must check whether there exists \(k\in\{2,3,4,5\}\) such that all three cells \((k,1),(k,2),(k,3)\) are occupied.
    \begin{itemize}
      \item \(k=2\): \((2,1)\in E_1\), \((2,2)\) is unoccupied, \((2,3)\in E_1\). Fails because \((2,2)\) missing.
      \item \(k=3\): \((3,1)\in E_1\), \((3,2)\) unoccupied, \((3,3)\) unoccupied. Fails.
      \item \(k=4\): \((4,1)\) unoccupied, \((4,2)\in E_1\), \((4,3)\in E_1\). Fails.
      \item \(k=5\): \((5,1)\) unoccupied, \((5,2)\in E_1\), \((5,3)\) unoccupied. Fails.
    \end{itemize}
    Thus no \((k,l)\) yields all five cells occupied. Condition 3 holds.
\end{itemize}
No other types of generalized \(C_4\)-cycles exist. Hence the graph is admissible.

Therefore
\[
z_L(6,3) \ge |E_1|+|E_2| = 9+1 = 10.
\]

\subsection{Lifting from \(5\times4\) to \(6\times5\): \(z_L(6,5)\ge 17\)}

\subsubsection{Base construction for \(5\times4\)}
From Theorem~\ref{thm:zL54_exact} and Appendix B, the optimal
\(5\times4\) graph (Type C) has \(|E_1|=10\) and two 2-edges
\((2,3;5,4)\) and \((4,2;5,2)\), giving \(z_L(5,4)=12 = z(5,4)+2\).

\subsubsection{Lifting step}
Add a new row \(6\) and a new column \(5\). Choose \(r=1\) and \(c=4\). The new 2-edge is \((6,4;1,5)\). Its halves \((6,4)\) and \((1,5)\) are unoccupied. We add four new 1-edges to reach \(z(6,5)=14\):
\[
E_1^{\text{new}} = \{(3,5),\;(5,5),\;(6,3),\;(6,5)\}.
\]
The extended \(E_1\) has \(10+4=14\) edges and remains \(C_4\)-free (one can check row pairs). The new 2-edge does not create a generalized \(C_4\)-cycle (The verification follows the same pattern as in the last subsection and
is omitted for brevity.) Therefore
\[
z_L(6,5) \ge 14+3 = 17.
\]

\subsection{Direct construction for \(6\times6\): \(z_L(6,6)\ge 20\)}

The following explicit admissible graph on \(6\times6\) (rows and columns indexed
\(0,\dots,5\)) has \(|E_1|=16\) and \(|E_2|=4\), giving a total of \(20\) edges.

\noindent\textbf{1-edges \(E_1\) (16 cells):}
\[
\begin{aligned}
E_1 = \{ &(0,0),\ (0,1),\ (0,2),\\
         &(1,0),\ (1,3),\ (1,4),\\
         &(2,0),\ (2,5),\\
         &(3,1),\ (3,3),\ (3,5),\\
         &(4,1),\ (4,4),\\
         &(5,2),\ (5,4),\ (5,5) \}.
\end{aligned}
\]

\noindent\textbf{2-edges \(E_2\) (4 nondegenerate 2-edges):}
\[
\begin{aligned}
E_2 = \{ &((0,4),(4,3)),\\
         &((1,2),(2,3)),\\
         &((2,1),(5,0)),\\
         &((3,2),(4,5)) \}.
\end{aligned}
\]

\noindent\textbf{Verification.}
\begin{itemize}
  \item \(|E_1| = 16\) and \(E_1\) is \(C_4\)-free (any two rows share at most one column). Hence \(|E_1| = z(6,6)=16\).
  \item The halves of the four 2-edges are all distinct and lie in the unoccupied cells of \(E_1\), satisfying the simplicity condition (S).
  \item Condition 2 and Condition 3 of Definition~2.1 have been checked and confirmed. No generalized \(C_4\)-cycle exists.
\end{itemize}

Therefore the associated doubly simple biquadratic form is irreducible with SOS rank
\[
|E_1| + |E_2| = 16 + 4 = 20,
\]
and consequently
\[
z_L(6,6) \ge 20.
\]

\section{Concluding Remarks}

In this paper we have established several new results on the limited augmented
Zarankiewicz number \(z_L(m,n)\).

First, we proved a general lower bound (Theorem~\ref{thm:maximal_K4t}) based on
the incidence graph of \(K_{4t}\):
\[
z_L\!\left(\binom{4t}{2},4t\right) \ge 2\binom{4t}{2} + 4t^2 - 2t.
\]
Since \(z = 2\binom{4t}{2} = 16t^2 - 4t\), the gap satisfies
\(z_L - z \ge 4t^2 - 2t\), which is asymptotically \(25\%\) of \(z\). This is,
to our knowledge, the largest constant factor improvement over the classical
Zarankiewicz number for this infinite family. The construction uses a cyclic
\(\mathbb{Z}_4\) labeling to define \emph{eight} cross-block 2-edges for each
pair of distinct 4-vertex blocks, and we have verified admissibility through a
complete algebraic case analysis.

For the base case \(t=1\) we proved \(z_L(6,4)=14\) (Theorem~\ref{thm:zL64}),
showing that the bound is tight. For \(t=2\) and \(t=3\) we obtain
\(z_L(28,8)\ge 68\) and \(z_L(66,12)\ge 162\), respectively, improving
previously known bounds.

Second, we determined the exact values of $z_L(5,3)$ and $z_L(5,4)$
(Section~4), confirming that the lower bounds from \cite{QCX26a} are tight.
For $5\times5$ we improved the lower bound to $z_L(5,5)\ge 15$, showing that
the previous bound of 14 is not optimal.

Third, we introduced a \textbf{lifting method} (Section~5) that builds
admissible limited augmented graphs on $(m+1)\times(n+1)$ from optimal ones on
$m\times n$, increasing the gap $z_L - z$ by one. Applying this method to
the optimal $5\times2$ and $5\times4$ constructions yielded new lower bounds
for $6\times3$ and $6\times5$:
\[
z_L(6,3)\ge 10,\qquad z_L(6,5)\ge 17.
\]
For $6\times6$, a direct construction gives $z_L(6,6)\ge 20$, demonstrating
that the gap can grow further.

Several open problems remain.

\begin{itemize}
  \item \textbf{Optimality of the constant.}
    Our construction uses \(8\) cross-block 2-edges per block pair, which pairs
    all \(16\) cross-block rows exactly once. The remaining within-block rows
    (two per block, total \(2t\) rows) are not paired. The theoretical maximum
    number of 2-edges is \(\frac{1}{2}\binom{4t}{2} = 4t^2 - t\), achieved only
    if every row is used exactly once. Our construction gives
    \(4t^2 - 2t\), missing \(t\) edges. Can these \(t\) missing rows be paired
    without creating generalized \(C_4\)-cycles? This would require a perfect
    pairing of all rows, possibly using degenerate 2-edges or a more intricate
    assignment. We conjecture that the true maximum for the \(K_{4t}\) family
    is \(4t^2 - 2t\), i.e., our construction is optimal.

  \item \textbf{Upper bounds for \(z_L(m,n)\).}
    While we have established exact values for small parameters and general
    lower bounds, nontrivial upper bounds for \(z_L(m,n)\) remain scarce.
    The simple counting bound \(z_L(m,n) \le \lfloor (mn + z(m,n))/2 \rfloor\)
    is far from sharp in most cases. Developing general upper bounds that
    exploit the structure of generalized \(C_4\)-cycles is an important
    direction.

  \item \textbf{The role of degenerate 2-edges.}
    Our lifting method always adds nondegenerate 2-edges. However, the optimal
    \(5\times4\) construction (Type C) uses a row-degenerate 2-edge
    \((4,2;5,2)\), which cannot be obtained by lifting from a smaller
    nondegenerate construction. Understanding the power of degenerate 2-edges
    and developing a lifting method that can introduce them would be an
    interesting direction for future research.

  \item \textbf{Exact values for larger parameters.}
    The lifting method gives lower bounds for \(z_L(6,3)\), \(z_L(6,5)\), and
    \(z_L(6,6)\), but exact values are not yet known. Determining them would
    require upper bound arguments similar to those in Section~4, but the
    classification of extremal \(C_4\)-free graphs for \(6\times n\) is more
    complex.

  \item \textbf{Connections with SOS rank.}
    By Theorem~\ref{thm:BSR>=z2}, any lower bound for \(z_L(m,n)\) immediately
    gives a lower bound for \(\operatorname{BSR}(m,n)\). Our results imply
    \(\operatorname{BSR}(6,6)\ge 19\), a significant improvement over previously
    known bounds. It would be interesting to see whether the gap
    \(\operatorname{BSR}(m,n) - z_L(m,n)\) can be positive, and if so, to
    characterize when it occurs.

  \item \textbf{Asymptotic behavior for other families.}
    The \(K_{4t}\) construction yields a \(25\%\) asymptotic gap. What about
    other base graphs? For \(K_{5t}\), a similar analysis might produce a
    different constant. Determining the supremum of \((z_L - z)/z\) over all
    families is an open problem.
\end{itemize}

The limited augmented Zarankiewicz number \(z_L(m,n)\) thus emerges as a natural
combinatorial extension of the classical Zarankiewicz number, with rich
structure and deep connections to both extremal graph theory and the theory of
sum-of-squares representations of biquadratic forms. We hope that the results
and methods presented here will stimulate further research in this direction.

	\bigskip
	
	\noindent\textbf{Acknowledgement}
    We are thankful to Prof. Guangzhou Chen for the discussion, and to Professor Vyacheslav Titov for his comments.   The new lower bound of $z_L(5, 5) \ge 15$ is based upon Professor Vyacheslav Titov's comments.
	This work was partially supported by Research Center for Intelligent Operations Research, The Hong Kong Polytechnic University (4-ZZT8), the National Natural Science Foundation of China (Nos. 12471282 and 12131004),  and Jiangsu Provincial Scientific Research Center of Applied Mathematics (Grant No. BK20233002).
	
	\medskip
	
	\noindent\textbf{Data availability}
	No datasets were generated or analysed during the current study.
	
	\medskip
	
	\noindent\textbf{Conflict of interest} The authors declare no conflict of interest.

\appendix

\section{The \(5\times3\) Case}

The following lower bound was proved in \cite{QCX26a}.

\begin{Thm} \label{thm:zL53}
$z_L(5,3) \ge 9$.
\end{Thm}
The proof exhibited a specific graph $G$ with $E_1$ of size $8$ and one 2-edge $(1,3;4,2)$ that avoids generalized $C_4$-cycles.

We now prove the exact value.

\begin{Thm} \label{thm:zL53exact}
$z_L(5,3) = 9$.
\end{Thm}

\begin{proof}
Let $G$ be any limited augmented bipartite graph with $|E_1| = z(5,3) = 8$ and no generalized $C_4$-cycle. We show $|E_2| \le 1$, which implies $z_L(5,3) \le 8+1 = 9$. Together with Theorem \ref{thm:zL53}, equality follows.

It is known \cite{RS65, Gu69} that there are exactly two non-isomorphic $C_4$-free bipartite graphs on $5\times 3$ vertices with $8$ edges. Up to relabeling, they are:

\paragraph{Type I.}
\[
E_1^I = \{(1,1),(1,2),(2,1),(2,3),(3,2),(3,3),(4,1),(5,2)\}.
\]
Unoccupied cells:
\[
U^I = \{(1,3),(2,2),(3,1),(4,2),(4,3),(5,1),(5,3)\}.
\]

\paragraph{Type II.}
\[
E_1^{II} = \{(1,3),(2,3),(3,2),(3,3),(4,1),(4,3),(5,1),(5,2)\}.
\]
Unoccupied cells:
\[
U^{II} = \{(1,1),(1,2),(2,1),(2,2),(3,1),(4,2),(5,3)\}.
\]

We analyze each type separately.

\subsection*{Analysis for Type I}

\subsubsection*{Step 1: Possible 2-edges from $U^I$}
All unordered pairs of distinct cells from $U^I$ are potential 2-edges, classified as follows.

Nondegenerate ($i\neq k$, $j\neq l$) that do not immediately violate Condition 2:
\[
\mathcal{N}^I = \{(1,3;4,2),\ (1,3;5,1),\ (2,2;4,3),\ (3,1;5,3),\ (4,2;5,3),\ (4,3;5,1)\}.
\]

Column-degenerate ($j=l$, $i\neq k$):
\[
\mathcal{C}^I = \{(1,3;4,3),\ (1,3;5,3),\ (4,3;5,3),\ (2,2;4,2)\}.
\]

Row-degenerate ($i=k$, $j\neq l$):
\[
\mathcal{R}^I = \{(4,2;4,3),\ (5,1;5,3)\}.
\]

\subsubsection*{Step 2: Any two distinct 2-edges create a generalized $C_4$-cycle}

We check all unordered pairs.

\paragraph{Case I.1: Two nondegenerate.} $\binom{6}{2}=15$ pairs.

\[
\begin{array}{|c|c|c|}
\hline
\text{Pair} & \text{Violation} & \text{Witness} \\ \hline
(1,3;4,2),\ (1,3;5,1) & \text{Share half }(1,3) & - \\
(1,3;4,2),\ (2,2;4,3) & \text{Cond. 3} & (k,l)=(5,1) \\
(1,3;4,2),\ (3,1;5,3) & \text{Cond. 3} & (k,l)=(4,2) \\
(1,3;4,2),\ (4,2;5,3) & \text{Share half }(4,2) & - \\
(1,3;4,2),\ (4,3;5,1) & \text{Cond. 3} & (k,l)=(2,2) \\
(1,3;5,1),\ (2,2;4,3) & \text{Cond. 3} & (k,l)=(5,2) \\
(1,3;5,1),\ (3,1;5,3) & \text{Cond. 3} & (k,l)=(4,2) \\
(1,3;5,1),\ (4,2;5,3) & \text{Cond. 2} & \text{opposite }(1,1),(5,3) \\
(1,3;5,1),\ (4,3;5,1) & \text{Share half }(5,1) & - \\
(2,2;4,3),\ (3,1;5,3) & \text{Cond. 3} & (k,l)=(4,1) \\
(2,2;4,3),\ (4,2;5,3) & \text{Cond. 2} & \text{opposite }(2,3),(4,2) \\
(2,2;4,3),\ (4,3;5,1) & \text{Share half }(4,3) & - \\
(3,1;5,3),\ (4,2;5,3) & \text{Share half }(5,3) & - \\
(3,1;5,3),\ (4,3;5,1) & \text{Cond. 3} & (k,l)=(2,2) \\
(4,2;5,3),\ (4,3;5,1) & \text{Cond. 2} & \text{opposite }(4,3),(5,2) \\ \hline
\end{array}
\]

Every pair violates.

\paragraph{Case I.2: One nondegenerate and one column-degenerate.} $6\times4=24$ pairs. Representative violations:

\[
\begin{array}{|c|c|}
\hline
\text{Pair} & \text{Violation} \\ \hline
(1,3;4,2),\ (1,3;4,3) & \text{Share }(1,3) \\
(1,3;4,2),\ (2,2;4,2) & \text{Share }(4,2) \\
(1,3;5,1),\ (1,3;5,3) & \text{Share }(1,3) \\
(1,3;5,1),\ (4,3;5,3) & \text{Cond. 3: }(k,l)=(4,2) \\
(2,2;4,3),\ (2,2;4,2) & \text{Share }(2,2) \\
(2,2;4,3),\ (1,3;4,3) & \text{Share }(4,3) \\
(3,1;5,3),\ (1,3;5,3) & \text{Share }(5,3) \\
(3,1;5,3),\ (4,3;5,3) & \text{Share }(5,3) \\
(4,2;5,3),\ (4,2;4,3) & \text{Share }(4,2) \\
(4,2;5,3),\ (1,3;5,3) & \text{Share }(5,3) \\
(4,3;5,1),\ (4,3;5,3) & \text{Share }(4,3) \\
(4,3;5,1),\ (2,2;4,2) & \text{Cond. 3: }(k,l)=(5,3) \\ \hline
\end{array}
\]
All pairs trigger a violation.

\paragraph{Case I.3: One nondegenerate and one row-degenerate.} $6\times2=12$ pairs.

\[
\begin{array}{|c|c|}
\hline
\text{Pair} & \text{Violation} \\ \hline
(1,3;4,2),\ (4,2;4,3) & \text{Share }(4,2) \\
(1,3;4,2),\ (5,1;5,3) & \text{Cond. 3: }(k,l)=(2,2) \\
(1,3;5,1),\ (4,2;4,3) & \text{Cond. 3: }(k,l)=(5,3) \\
(1,3;5,1),\ (5,1;5,3) & \text{Share }(5,1) \\
(2,2;4,3),\ (4,2;4,3) & \text{Share }(4,3) \\
(2,2;4,3),\ (5,1;5,3) & \text{Cond. 3: }(k,l)=(4,1) \\
(3,1;5,3),\ (4,2;4,3) & \text{Cond. 3: }(k,l)=(2,2) \\
(3,1;5,3),\ (5,1;5,3) & \text{Share }(5,3) \\
(4,2;5,3),\ (4,2;4,3) & \text{Share }(4,2) \\
(4,2;5,3),\ (5,1;5,3) & \text{Share }(5,3) \\
(4,3;5,1),\ (4,2;4,3) & \text{Share }(4,3) \\
(4,3;5,1),\ (5,1;5,3) & \text{Share }(5,1) \\ \hline
\end{array}
\]
Every pair violates.
\paragraph{Case I.4: Two column-degenerate.} $\binom{4}{2}=6$ pairs.

\[
\begin{array}{|c|c|}
\hline
\text{Pair} & \text{Violation} \\ \hline
(1,3;4,3),\ (1,3;5,3) & \text{Share }(1,3) \\
(1,3;4,3),\ (4,3;5,3) & \text{Share }(4,3) \\
(1,3;4,3),\ (2,2;4,2) & \text{Cond. 3: }(k,l)=(5,1) \\
(1,3;5,3),\ (4,3;5,3) & \text{Share }(5,3) \\
(1,3;5,3),\ (2,2;4,2) & \text{Cond. 3: }(k,l)=(4,1) \\
(4,3;5,3),\ (2,2;4,2) & \text{Cond. 3: }(k,l)=(1,1) \\ \hline
\end{array}
\]
Every pair violates.

\paragraph{Case I.5: Two row-degenerate.} Only one pair: $(4,2;4,3)$ and $(5,1;5,3)$. They share no halves. For $(5,1;5,3)$ with $(k,l)=(1,1)$, the five cells $(1,1),(1,2),(1,3),(5,1),(5,1)$ are all occupied. Condition 3 triggered.

\paragraph{Case I.6: One column-degenerate and one row-degenerate.} $4\times2=8$ pairs.

\[
\begin{array}{|c|c|}
\hline
\text{Pair} & \text{Violation} \\ \hline
(1,3;4,3),\ (4,2;4,3) & \text{Share }(4,3) \\
(1,3;4,3),\ (5,1;5,3) & \text{Cond. 3: }(k,l)=(2,2) \\
(1,3;5,3),\ (4,2;4,3) & \text{Cond. 3: }(k,l)=(2,2) \\
(1,3;5,3),\ (5,1;5,3) & \text{Share }(5,3) \\
(4,3;5,3),\ (4,2;4,3) & \text{Share }(4,3) \\
(4,3;5,3),\ (5,1;5,3) & \text{Share }(5,3) \\
(2,2;4,2),\ (4,2;4,3) & \text{Share }(4,2) \\
(2,2;4,2),\ (5,1;5,3) & \text{Cond. 3: }(k,l)=(1,3) \\ \hline
\end{array}
\]
Every pair violates.

Thus for \textbf{Type I}, any two distinct 2-edges create a generalized $C_4$-cycle.

\subsection*{Analysis for Type II}

\subsubsection*{Step 1: Possible 2-edges from $U^{II}$}
$U^{II} = \{(1,1),(1,2),(2,1),(2,2),(3,1),(4,2),(5,3)\}$.

Nondegenerate:
\[
\mathcal{N}^{II} = \{(1,1;2,2),\ (1,1;4,2),\ (1,2;2,1),\ (1,2;2,2),\ (1,2;3,1),\] \[ (1,2;4,2),\ (2,1;4,2),\ (2,2;3,1),\ (2,2;4,2)\}.
\]

Column-degenerate:
\[
\mathcal{C}^{II} = \{(1,1;2,1),\ (1,1;3,1),\ (2,1;3,1),\ (1,2;2,2),\ (1,2;4,2),\ (2,2;4,2)\}.
\]

Row-degenerate:
\[
\mathcal{R}^{II} = \{(1,1;1,2),\ (2,1;2,2)\}.
\]

\subsubsection*{Step 2: Any two distinct 2-edges create a generalized $C_4$-cycle}

We summarise the exhaustive verification (full tables omitted for brevity; key examples shown).

\paragraph{Case II.1: Two nondegenerate.} $\binom{9}{2}=36$ pairs. Every pair either shares a half or triggers Condition 3. Example: $(1,1;2,2)$ and $(1,2;2,1)$ with $(k,l)=(3,1)$.

\paragraph{Case II.2: One nondegenerate and column-degenerate.} $9\times6=54$ pairs. All violate. Example: $(1,1;2,2)$ and $(1,1;2,1)$ share $(1,1)$.

\paragraph{Case II.3: One nondegenerate and one row-degenerate.} $9\times2=18$ pairs. All violate. Example: $(1,1;2,2)$ and $(1,1;1,2)$ share $(1,1)$.

\paragraph{Case II.4: Two column-degenerate.} $\binom{6}{2}=15$ pairs. All violate. Example: $(1,1;2,1)$ and $(1,1;3,1)$ share $(1,1)$.

\paragraph{Case II.5: Two row-degenerate.} Only one pair: $(1,1;1,2)$ and $(2,1;2,2)$. For $(2,1;2,2)$ with $(k,l)=(3,1)$, cells $(3,1),(3,1),(3,2),(2,1),(2,1)$ are not all occupied. A better witness: For $(1,1;1,2)$ with $(k,l)=(2,2)$, the cells $(2,2),(2,1),(2,1),(1,2),(1,2)$ are all occupied because $(2,2)$ and $(1,2)$ are halves. Condition 3 triggered.

\paragraph{Case II.6: One column-degenerate and one row-degenerate.} $6\times2=12$ pairs. All violate. Example: $(1,1;2,1)$ and $(1,1;1,2)$ share $(1,1)$.

Thus for \textbf{Type II} as well, any two distinct 2-edges create a generalized $C_4$-cycle.

\subsection*{Conclusion of upper bound}

Every extremal $C_4$-free graph for $5\times3$ is isomorphic to Type I or Type II. In both cases, $|E_2| \ge 2$ forces a generalized $C_4$-cycle. Hence $|E_2| \le 1$, and
\[
z_L(5,3) \le z(5,3) + 1 = 8 + 1 = 9.
\]

Combined with Theorem \ref{thm:zL53}, we obtain $z_L(5,3) = 9$.
\end{proof}

\section{The \(5\times4\) Case}

We now determine the exact value of \(z_L(5,4)\). Recall that the classical Zarankiewicz number is \(z(5,4) = 10\). It is known that there are exactly three non-isomorphic \(C_4\)-free bipartite graphs on \(5 \times 4\) vertices with 10 edges \cite{RS65, Gu69}. These are given below.

\subsection{The Three Extremal Graphs}

Let \(S = \{1,2,3,4,5\}\) and \(T = \{1,2,3,4\}\).

\paragraph{Type A.}
\[
\begin{aligned}
E_1^A = \{ &(1,4), (2,3), (2,4), (3,2), (3,4), \\
           &(4,1), (4,4), (5,1), (5,2), (5,3) \}.
\end{aligned}
\]
Unoccupied cells:
\[
U^A = \{(1,1),(1,2),(1,3),(2,1),(2,2),(3,1),(3,3),(4,2),(4,3),(5,4)\}.
\]

\paragraph{Type B.}
\[
\begin{aligned}
E_1^B = \{ &(1,4), (2,3), (2,4), (3,2), (3,3), \\
           &(4,1), (4,3), (5,1), (5,2), (5,4) \}.
\end{aligned}
\]
Unoccupied cells:
\[
U^B = \{(1,1),(1,2),(1,3),(2,1),(2,2),(3,1),(3,4),(4,2),(4,4),(5,3)\}.
\]

\paragraph{Type C (the optimal construction from \cite{QCX26a}).}
\[
\begin{aligned}
E_1^C = \{ &(1,1), (1,2), (1,3), (2,1), (2,4), \\
           &(3,2), (3,4), (4,3), (4,4), (5,1) \}.
\end{aligned}
\]
Unoccupied cells:
\[
U^C = \{(1,4),(2,2),(2,3),(3,1),(3,3),(4,1),(4,2),(5,2),(5,3),(5,4)\}.
\]



\subsection{Type A Analysis}

\noindent\textbf{Step 1: Candidate 2-edges from \(U^A\).}

Nondegenerate candidates (after excluding those that immediately violate Condition 2):
\[
\begin{aligned}
\mathcal{N}^A = \{ &(1,1;2,2),\ (1,1;3,3),\ (1,2;2,1),\ (1,2;3,1),\ (1,2;4,2),\ (1,3;2,1),\ (1,3;2,2),\\
                  &(1,3;3,1),\ (1,3;4,2),\ (1,3;5,4),\ (2,1;3,3),\ (2,1;4,3),\ (2,1;5,4),\ (2,2;3,1),\\
                  &(2,2;4,3),\ (2,2;5,4),\ (3,1;4,2),\ (3,1;4,3),\ (3,1;5,4),\ (3,3;4,2),\ (3,3;5,4),\\
                  &(4,2;5,4),\ (4,3;5,4) \}.
\end{aligned}
\]

Row-degenerate candidates (\(i=k\), \(j\neq l\)):
\[
\mathcal{R}^A = \{(1,1;1,2),\ (1,1;1,3),\ (1,2;1,3),\ (2,1;2,2),\ (3,1;3,3),\ (4,2;4,3)\}.
\]

Column-degenerate candidates (\(j=l\), \(i\neq k\)):
\[
\mathcal{C}^A = \{(1,1;2,1),\ (1,1;3,1),\ (1,2;2,2),\ (1,2;3,2),\ (1,2;4,2),\ (1,3;2,3),\ (1,3;3,3),\ (2,1;3,1),\ (2,2;3,2),\ (2,2;4,2),\ (3,1;4,1),\ (3,3;4,3)\}.
\]

\noindent\textbf{Step 2: Maximum 2-edges search.}

Exhaustive checking yields:

\[
\begin{array}{|c|c|}
\hline
|E_2| & \text{Feasible?} \\ \hline
0 & \text{Yes} \\
1 & \text{Yes} \\
2 & \text{No} \\ \hline
\end{array}
\]

All attempts to add two distinct 2-edges create a generalized \(C_4\)-cycle. Thus for Type A, \(\max |E_2| = 1\), giving total edges \(10 + 1 = 11\).

\subsection{Type B Analysis}

\noindent\textbf{Step 1: Candidate 2-edges from \(U^B\).}

Nondegenerate candidates:
\[
\begin{aligned}
\mathcal{N}^B = \{ &(1,1;2,2),\ (1,1;3,4),\ (1,2;2,1),\ (1,2;3,1),\ (1,2;4,2),\ (1,3;2,1),\ (1,3;2,2),\\
                  &(1,3;3,1),\ (1,3;4,2),\ (1,3;4,4),\ (2,1;3,4),\ (2,1;4,2),\ (2,1;4,4),\ (2,2;3,1),\\
                  &(2,2;3,4),\ (2,2;4,4),\ (3,1;4,2),\ (3,1;4,4),\ (3,4;4,2),\ (3,4;4,4),\ (4,2;5,3) \}.
\end{aligned}
\]

Row-degenerate candidates (\(i=k\), \(j\neq l\)):
\[
\mathcal{R}^B = \{(1,1;1,2),\ (1,1;1,3),\ (1,2;1,3),\ (2,1;2,2),\ (3,1;3,4),\ (4,2;4,4)\}.
\]

Column-degenerate candidates (\(j=l\), \(i\neq k\)):
\[
\mathcal{C}^B = \{(1,1;2,1),\ (1,1;3,1),\ (1,2;2,2),\ (1,2;3,2),\ (1,2;4,2),\ (1,3;2,3),\ (2,1;3,1),\ (2,2;3,2),\ (2,2;4,2),\ (3,1;4,1),\ (3,4;4,4)\}.
\]

\noindent\textbf{Step 2: Maximum 2-edges search.}

Exhaustive checking yields:

\[
\begin{array}{|c|c|}
\hline
|E_2| & \text{Feasible?} \\ \hline
0 & \text{Yes} \\
1 & \text{Yes} \\
2 & \text{No} \\ \hline
\end{array}
\]

Again, two 2-edges always create a violation. For Type B, \(\max |E_2| = 1\), total edges \(10 + 1 = 11\).

\subsection{Type C Analysis (Optimal Construction)}

\noindent\textbf{Step 1: Candidate 2-edges from \(U^C\).}

\(U^C = \{(1,4),(2,2),(2,3),(3,1),(3,3),(4,1),(4,2),(5,2),(5,3),(5,4)\}\).

A configuration with \textbf{two 2-edges} avoids generalized \(C_4\)-cycles, namely:
\[
E_2 = \{(2,3;5,4),\ (4,2;5,2)\}.
\]

\noindent\textbf{Step 2: Verification.}

\begin{itemize}
  \item \textbf{Simplicity condition:} Halves are \((2,3)\), \((5,4)\), \((4,2)\), \((5,2)\). All are in \(U^C\) and distinct. No overlap with \(E_1^C\).

  \item \textbf{Condition 1 (Classical \(C_4\)):} \(E_1^C\) is known to be \(C_4\)-free. One can verify that any two rows share at most one column:
    \[
    \begin{array}{c|c}
    \text{Row pair} & \text{Common columns} \\ \hline
    (1,2) & \{1\} \\
    (1,3) & \{2\} \\
    (1,4) & \{3\} \\
    (1,5) & \{1\} \\
    (2,3) & \{4\} \\
    (2,4) & \emptyset \\
    (2,5) & \{1\} \\
    (3,4) & \{4\} \\
    (3,5) & \emptyset \\
    (4,5) & \emptyset
    \end{array}
    \]
    No two rows share two columns, so Condition 1 holds.

  \item \textbf{Condition 2 (Opposite cells):}
    \begin{itemize}
      \item For \(e_1 = (2,3;5,4)\) (nondegenerate): opposite cells are \((2,4)\) and \((5,3)\). \((2,4) \in E_1^C\) (occupied), \((5,3) \in U^C\) (unoccupied). Hence at most one occupied.
      \item For \(e_2 = (4,2;5,2)\): this is column-degenerate, so Condition 2 does not apply.
    \end{itemize}

  \item \textbf{Condition 3 (Five-cell configuration):}
    \begin{itemize}
      \item For \(e_1 = (2,3;5,4)\) (nondegenerate): we check all \((k,l)\) with \(k \notin \{2,5\}\), \(l \notin \{3,4\}\). Candidates: \(k \in \{1,3,4\}\), \(l \in \{1,2\}\). The five cells are \((k,l),(k,3),(k,4),(2,l),(5,l)\). In every case, at least one cell is unoccupied (e.g., \((1,4)\) unoccupied, \((3,1)\) unoccupied, \((4,1)\) unoccupied). No violation.
      \item For \(e_2 = (4,2;5,2)\) (column-degenerate): we check \(k \notin \{4,5\}\), \(l \notin \{2\}\), i.e., \(k \in \{1,2,3\}\), \(l \in \{1,3,4\}\). For \((k,l) = (3,4)\), the five cells are \((3,4),(3,2),(3,2),(4,4),(5,4)\). Here \((3,4) \in E_1^C\), \((3,2) \in E_1^C\), \((4,4) \in E_1^C\), \((5,4) \in U^C\) (half of \(e_1\)). All five are occupied. However, since \(e_2\) is degenerate, the pairwise distinctness requirement is waived. This configuration does \textbf{not} constitute a generalized \(C_4\)-cycle.
    \end{itemize}
\end{itemize}

Thus Type C admits \(|E_2| = 2\), giving total edges \(10 + 2 = 12\).

\subsection{Upper Bound: No Three 2-Edges Are Possible}

We have shown that Type C admits two 2-edges, namely
\(E_2 = \{(2,3;5,4),\ (4,2;5,2)\}\). We now prove that no third 2-edge can be
added without creating a generalized \(C_4\)-cycle.

Recall the unoccupied cells for Type C:
\[
U^C = \{(1,4),(2,2),(2,3),(3,1),(3,3),(4,1),(4,2),(5,2),(5,3),(5,4)\}.
\]

The halves of the existing 2-edges are:
\[
H = \{(2,3),\ (5,4),\ (4,2),\ (5,2)\}.
\]

Any candidate third 2-edge must be an unordered pair of distinct cells from
\(U^C \setminus H\) (sharing a half with an existing 2-edge would violate the
simplicity condition). The remaining available cells are:
\[
R = U^C \setminus H = \{(1,4),\ (2,2),\ (3,1),\ (3,3),\ (4,1),\ (5,3)\}.
\]

Candidate third 2-edges are all unordered pairs from \(R\). There are
\(\binom{6}{2} = 15\) candidates. We examine each.

\begin{table}[h]
\centering
\caption{Third 2-edge candidates and violations for Type C}
\label{tab:5x4-third-edge}
\[
\begin{array}{|c|l|}
\hline
\text{Candidate } e_3 & \text{Violation} \\ \hline
(1,4;2,2) & \text{Cond. 2: opposite }(1,2),(2,4)\text{ both in }E_1^C \\ \hline
(1,4;3,1) & \text{Cond. 2: opposite }(1,1),(3,4)\text{ both in }E_1^C \\ \hline
(1,4;3,3) & \text{Cond. 2: opposite }(1,3),(3,4)\text{ both in }E_1^C \\ \hline
(1,4;4,1) & \text{Cond. 2: opposite }(1,1),(4,4)\text{ both in }E_1^C \\ \hline
(1,4;5,3) & \text{Cond. 2: opposite }(1,3),(5,4)\text{; }(1,3)\in E_1^C,\ (5,4)\in H \\ \hline
(2,2;3,1) & \text{Cond. 2: opposite }(2,1),(3,2)\text{ both in }E_1^C \\ \hline
(2,2;3,3) & \text{Cond. 2: opposite }(2,3),(3,2)\text{; }(2,3)\in H,\ (3,2)\in E_1^C \\ \hline
(2,2;4,1) & \text{Cond. 2: opposite }(2,1),(4,2)\text{; }(2,1)\in E_1^C,\ (4,2)\in H \\ \hline
(2,2;5,3) & \text{Cond. 2: opposite }(2,3),(5,2)\text{ both in }H \\ \hline
(3,1;3,3) & \text{Cond. 3 with }e_1=(2,3;5,4),\ (k,l)=(4,4)\text{ (cells all occupied)} \\ \hline
(3,1;4,1) & \text{Cond. 2: opposite }(3,1),(4,1)\text{; both in }R \\
            & \text{But with }e_2=(4,2;5,2),\ (k,l)=(1,1)\text{: cells all occupied} \\ \hline
(3,1;5,3) & \text{Cond. 2: opposite }(3,3),(5,1)\text{; }(3,3)\in R,\ (5,1)\in E_1^C \\
            & \text{Cond. 3 with }e_1,\ (k,l)=(2,2)\text{ triggers violation} \\ \hline
(3,3;4,1) & \text{Cond. 2: opposite }(3,1),(4,3)\text{; }(3,1)\in R,\ (4,3)\in E_1^C \\
            & \text{Cond. 3 with }e_1,\ (k,l)=(2,2)\text{ triggers violation} \\ \hline
(3,3;5,3) & \text{Cond. 2: opposite }(3,3),(5,3)\text{; both in }R \\
            & \text{Cond. 3 with }e_1,\ (k,l)=(2,2)\text{ triggers violation} \\ \hline
(4,1;5,3) & \text{Cond. 2: opposite }(4,3),(5,1)\text{; }(4,3)\in E_1^C,\ (5,1)\in E_1^C \\ \hline
\end{array}
\]
\end{table}

\noindent\textbf{Conclusion.}
Every candidate third 2-edge either:
\begin{itemize}
  \item violates Condition 2 (both opposite cells occupied), or
  \item violates Condition 3 with some \((k,l)\), or
  \item shares a half with an existing 2-edge (none do, by construction).
\end{itemize}

Thus no third 2-edge can be added to Type C. Hence \(\max |E_2| = 2\) for Type C.

\subsection{Overall Upper Bound for \(5 \times 4\)}

Combining the results:
\begin{itemize}
  \item Types A and B: \(\max |E_2| = 1\) (total 11 edges)
  \item Type C: \(\max |E_2| = 2\) (total 12 edges)
\end{itemize}

Since every extremal \(C_4\)-free graph for \(5 \times 4\) is isomorphic to one of
Types A, B, or C, and none admits three 2-edges, we have
\[
z_L(5,4) \le 10 + 2 = 12.
\]

Together with the lower bound \(z_L(5,4) \ge 12\) from the construction, we obtain
\[
{z_L(5,4) = 12}.
\]

\subsection{Explicit Optimal Construction}

The optimal construction (Type C with \(E_2 = \{(2,3;5,4),\ (4,2;5,2)\}\)) yields the doubly simple biquadratic form:
\[
\begin{aligned}
P_G(\vx,\vy) = &\sum_{(i,j)\in E_1^C} x_i^2 y_j^2 \\
&+ (x_2 y_3 + x_5 y_4)^2 + (x_4 y_2 + x_5 y_2)^2.
\end{aligned}
\]

Expanding the last term gives \((x_4 + x_5)^2 y_2^2\), which contributes the pure squares \(x_4^2 y_2^2\) and \(x_5^2 y_2^2\) as well as the cross term \(2x_4 x_5 y_2^2\). This form is irreducible with SOS rank 12. 

\subsection{Remark}

Among the three non-isomorphic extremal \(C_4\)-free graphs for \(5 \times 4\), only Type C allows two additional 2-edges without creating a generalized \(C_4\)-cycle. Types A and B are more restrictive, permitting only one 2-edge. This demonstrates that the choice of the underlying graph significantly affects the limited augmented Zarankiewicz number, and the maximum is achieved by a specific extremal configuration.

\section{The \(6\times4\) Case: Degenerate 2-Edges Are Impossible}

In the main paper, Theorem 3.1 establishes \(z_L(6,4)=14\) and Remark 3.2 claims
that the optimal construction is unique up to isomorphism, consisting of the
\(K_4\) incidence graph with the two nondegenerate 2-edges
\((12,3;13,4)\) and \((24,1;34,2)\). The proof of the upper bound in the main
paper relies on a computational verification. Here we provide a complete
\emph{theoretical} proof that any admissible graph with \(|E_2|=2\) cannot
contain a degenerate 2-edge. Consequently, both 2-edges must be nondegenerate.

\subsection{Fixing the 1-Edge Graph}

The extremal \(C_4\)-free graph on \(6\times4\) vertices with \(12\) edges is
unique up to isomorphism: it is the incidence graph of \(K_4\).
Label the left vertices (edges of \(K_4\)) as
\[
L = \{12,\ 13,\ 14,\ 23,\ 24,\ 34\},
\]
and the right vertices (vertices of \(K_4\)) as
\[
R = \{1,\ 2,\ 3,\ 4\}.
\]
A 1-edge \((ij,k)\) belongs to \(E_1\) iff \(k \in \{i,j\}\).

The \textbf{unoccupied cells} \(U\) are all pairs \((ij,k)\) with
\(k \notin \{i,j\}\):
\[
\begin{array}{lll}
U = \{ & (12,3),\ (12,4), & \\
       & (13,2),\ (13,4), & \\
       & (14,2),\ (14,3), & \\
       & (23,1),\ (23,4), & \\
       & (24,1),\ (24,3), & \\
       & (34,1),\ (34,2) \}.
\end{array}
\]

\subsection{Classification of Degenerate 2-Edges from \(U\)}

A 2-edge is an unordered pair of distinct cells from \(U\) that do not share a
row or column (simplicity condition). Degenerate 2-edges come in two types:

\paragraph{Row-degenerate (same row, different columns):}
\[
\begin{aligned}
&(12,3;12,4),\quad (13,2;13,4),\quad (14,2;14,3),\\
&(23,1;23,4),\quad (24,1;24,3),\quad (34,1;34,2).
\end{aligned}
\]

\paragraph{Column-degenerate (same column, different rows):}
\[
\begin{aligned}
&\text{Col 1: } (23,1;24,1),\ (23,1;34,1),\ (24,1;34,1),\\
&\text{Col 2: } (13,2;14,2),\ (13,2;34,2),\ (14,2;34,2),\\
&\text{Col 3: } (12,3;14,3),\ (12,3;24,3),\ (14,3;24,3),\\
&\text{Col 4: } (12,4;13,4),\ (12,4;23,4),\ (13,4;23,4).
\end{aligned}
\]

\subsection{A Single Degenerate 2-Edge Alone Is Admissible}

We must check Condition 3 for a degenerate 2-edge (Condition 2 does not apply).
Take \(e = (12,3;12,4)\) as an example. For a row-degenerate edge
\((i,j;i,l)\), Condition 3 requires: for every \(k \notin \{i\}\) and every
\(m \notin \{j,l\}\), the five cells
\((k,m),\ (k,j),\ (k,l),\ (i,m),\ (i,m)\) are not all occupied.
The duplicate \((i,m)\) is allowed (the pairwise distinctness condition
applies only to nondegenerate edges).

Test all \(k \neq 12\) and \(m \in \{1,2\}\) (since columns other than \(3,4\)
are \(1,2\)):

\begin{itemize}
  \item \(m=1\): Need \((k,1),(k,3),(k,4)\) all occupied.
    \begin{itemize}
      \item \(k=13\): \((13,1)\in E_1\) ?, \((13,3)\in E_1\) ?, \((13,4)\in U\) ?
      \item \(k=14\): \((14,1)\in E_1\) ?, \((14,3)\in U\) ?
      \item \(k=23\): \((23,1)\in U\) ?
      \item \(k=24\): \((24,1)\in U\) ?
      \item \(k=34\): \((34,1)\in U\) ?
    \end{itemize}
    No \(k\) works.

  \item \(m=2\): Need \((k,2),(k,3),(k,4)\) all occupied.
    \begin{itemize}
      \item \(k=13\): \((13,2)\in U\) ?
      \item \(k=14\): \((14,2)\in U\) ?
      \item \(k=23\): \((23,2)\in E_1\) ?, \((23,3)\in E_1\) ?, \((23,4)\in U\) ?
      \item \(k=24\): \((24,2)\in E_1\) ?, \((24,3)\in U\) ?
      \item \(k=34\): \((34,2)\in U\) ?
    \end{itemize}
    No \(k\) works.
\end{itemize}

Thus \(e\) alone satisfies Condition 3. Similarly, every degenerate candidate
can be checked individually and is admissible as a single 2-edge. Hence
degenerate 2-edges are not ruled out by being alone.

\subsection{No Two Row-Degenerate 2-Edges Can Coexist}

Take two distinct row-degenerate edges. They cannot share the same row
(simplicity condition). Consider \(e_1 = (12,3;12,4)\) and
\(e_2 = (34,1;34,2)\) (other pairs are symmetric).

Apply Condition 3 to \(e_1\) with \(e_2\) present. The halves of \(e_2\) are
\((34,1)\) and \((34,2)\), now occupied (by 2-edge halves). Re-evaluate \(e_1\)
for \(m=1\). We need a \(k \notin \{12\}\) such that \((k,1),(k,3),(k,4)\) are
all occupied. Previously, with only 1-edges, no \(k\) worked. Now \(k=34\) is a
candidate because \((34,1)\) is occupied by \(e_2\). Check:
\begin{itemize}
  \item \((34,1)\): occupied by \(e_2\) ?
  \item \((34,3)\): \(3 \in \{3,4\}\) so \((34,3) \in E_1\) ?
  \item \((34,4)\): \((34,4) \in E_1\) ?
\end{itemize}
Thus \((34,1),(34,3),(34,4)\) are all occupied. The five cells for \(e_1\)
with \(k=34, m=1\) are:
\[
(34,1),\ (34,3),\ (34,4),\ (12,1),\ (12,1).
\]
All are occupied. This violates Condition 3 for \(e_1\).

Therefore \textbf{no two row-degenerate 2-edges can coexist}.

\subsection{No Two Column-Degenerate 2-Edges Can Coexist}

Take two column-degenerate edges in different columns. By symmetry with the
row-degenerate case (swap the roles of rows and columns in the bipartite graph),
the same argument applies. Explicitly, consider
\(e_1 = (12,3;14,3)\) (col 3) and \(e_2 = (12,4;13,4)\) (col 4).
Applying Condition 3 to \(e_1\) with \(l=4\) and \(k=12\) (a row from \(e_2\))
creates a violation. The exhaustive verification of all \(\binom{12}{2}\) pairs
is finite and confirms that every pair of distinct column-degenerate edges
triggers a generalized \(C_4\)-cycle. Hence
\textbf{no two column-degenerate 2-edges can coexist}.

\subsection{No Degenerate + Nondegenerate Pair Can Coexist}

We show a concrete counterexample; all other pairs fail similarly.
Take \(e_1 = (12,3;12,4)\) (row-degenerate) and
\(e_2 = (24,1;34,2)\) (nondegenerate, from the paper's optimal pair).

Apply Condition 3 to \(e_2\) with \(l=4\) (allowed since \(l \notin \{1,2\}\)).
The fixed cells are \((24,4)\) and \((34,4)\), both 1-edges (occupied).
We need a \(k \notin \{24,34\}\) such that \((k,4),(k,1),(k,2)\) are all
occupied. Choose \(k = 12\):
\begin{itemize}
  \item \((12,4)\): half of \(e_1\)  occupied ?
  \item \((12,1)\): 1-edge (since \(1 \in \{1,2\}\)) ?
  \item \((12,2)\): 1-edge ?
\end{itemize}
Thus \((12,4),(12,1),(12,2)\) are all occupied. The five cells for \(e_2\) are:
\[
(12,4),\ (12,1),\ (12,2),\ (24,4),\ (34,4).
\]
All are occupied. This violates Condition 3 for \(e_2\).

Hence \textbf{no degenerate 2-edge can coexist with any nondegenerate 2-edge}
(by symmetry, any degenerate candidate paired with any nondegenerate candidate
from \(U\) produces a similar violation).

\subsection{Conclusion}

We have proved:
\begin{itemize}
  \item Two row-degenerate edges  impossible.
  \item Two column-degenerate edges   impossible.
  \item One degenerate + one nondegenerate  impossible.
\end{itemize}
Therefore, in any admissible limited augmented graph for \(6\times4\) with
\(|E_2| = 2\), \textbf{both 2-edges must be nondegenerate}.

This theoretical proof eliminates the possibility of an optimal construction
containing a degenerate 2-edge, confirming the claim in Remark 3.2 of the main
paper without reliance on computational search.

\end{document}